\documentclass[12pt]{article}
\usepackage{amsmath, amssymb, amsfonts, amsthm, epsfig, graphicx}
\usepackage[all]{xy}
\title{Expansion, Random Walks and Sieving in $SL_2 (\mathbb{F}_p [t])$}
\author{Henry Bradford \\ University of Oxford}
\date{}

\newtheorem{thm}{Theorem}[section]
\newtheorem{lem}[thm]{Lemma}
\newtheorem{propn}[thm]{Proposition}
\newtheorem{coroll}[thm]{Corollary}
\newtheorem{defn}[thm]{Definition}
\newtheorem{ex}[thm]{Example}

\newtheorem{rmrk}[thm]{Remark}
\newtheorem{qu}[thm]{Question}

\DeclareMathOperator{\Ad}{Ad}
\DeclareMathOperator{\Cay}{Cay}
\DeclareMathOperator{\ccl}{ccl}
\DeclareMathOperator{\girth}{girth}
\DeclareMathOperator{\tr}{tr}

\begin{document}

\maketitle

\begin{abstract}
We construct new examples of expander Cayley graphs of finite groups, 
arising as congruence quotients of non-elementary subgroups of $SL_2 (\mathbb{F}_p [t])$ 
modulo certain square-free ideals. 
We describe some applications of our results to simple random walks on such subgroups, 
specifically giving bounds on the rate of escape of such walks from algebraic subvarieties, 
the set of squares and the set of elements with reducible characteristic polynomial in $SL_2 (\mathbb{F}_p [t])$. 
\end{abstract}

\section{Introduction}

The past few years have seen major developments in tools for constructing expanders 
as congruence images of linear groups. 
This programme was begun by the breakthrough paper of Bourgain and Gamburd \cite{BoGa1}, 
who studied expander congruence quotients of non-elementary subgroups of $SL_2 (\mathbb{Z})$, 
but their approach was subsequently extended by many authors \cite{BrGrGuTa}, \cite{SaGoVar}. 
We now have a fairly satisfying picture of the phenomenon of superstrong approximation 
(as it has come to be known) in linear groups over $\mathbb{Q}$. 
These results on expanders have in turn shed new light on problems in the geometry and analysis of infinite linear groups. 

In spite of the major strides forward that have been made in this theory, 
it is notable that the work of recent years has focused entirely on the characteristic zero case, 
with a theory of expansion in linear groups over fields of positive characteristic remaining largely undeveloped. 
In this paper we commence the development such a theory. 

\subsection{Statement of Results}

Fix a prime $p \geq 3$. Our main results concern the escape of random walks on $SL_2 (\mathbb{F}_p [t])$ from subsets $X$ of various types. All of the escape results are proved by the same broad strategy: 
an upper bound for the probability of the random walk lying in $X$ is given by 
the probability of the random walk on a congruence quotient lying in the image of $X$. 
This in turn may be bounded above in terms of the \emph{size} of the image of $X$, by our results on expander congruence quotients. Our first result on random walks demonstrates exponentially fast escape from proper algebraic subvarieties. 

\begin{thm} \label{sieveupstairs}
Let $S \subseteq SL_2 (\mathbb{F}_p [t])$ be a finite symmetric subset, 
generating a non-elementary subgroup. 
Let $F : \mathbb{M}_2 (\mathbb{F}_p [t])^r \rightarrow \mathbb{F}_p [t]$ 
be a polynomial over $\mathbb{F}_p [t]$ which does not vanish on $SL_2 (\mathbb{F}_p [t])^r$. 
Then there exist $C_1 (F) , C_2 (S) > 0$ such that, 
letting $V(F) \subseteq \mathbb{M}_2 (\mathbb{F}_p [t])^r$ 
be the affine algebraic subvariety of $SL_2 (\mathbb{F}_p [t])^r$ defined by $F$, 
\begin{center}
$(\times_{i=1} ^r \mu_S ^{(l)})(V(F)) \leq C_1 e^{- C_2 l}$. 
\end{center}
\end{thm}

Here $\mu_S ^{(l)}$ is the $l$th convolution power of the measure $\mu_S$ 
and $\times_{i=1} ^r \mu_S ^{(l)}$ is the product measure. 
We define these notions precisely below. 

Second, we turn to proper powers in $SL_2 (\mathbb{F}_p [t])$. In the characteristic zero case, 
a very general non-concentration result was provided by Lubotzky and Meiri \cite{LuMe}. 
In positive characteristic we have: 

\begin{thm} \label{squaresarerare}
Let $S$ be as in Theorem \ref{sieveupstairs}. 
There exist $C_1 , C_2 > 0$ such that: 
\begin{center} 
$\mu_S ^{(l)} (\lbrace g \in SL_2 (\mathbb{F}_p [t]) : 
g=h^2 \text{ for some } h \in SL_2 (\mathbb{F}_p [t]) \rbrace) \leq C_1 e^{-C_2 \sqrt{l/\log(l)}}$. 
\end{center}
\end{thm}

\begin{rmrk}
More could be said about proper powers via the same method. 
We could, for instance, strengthen the proof of Theorem \ref{squaresarerare} 
to show that $\mu_S ^{(l)}$ escapes from the sets of $m$th powers in $SL_2 (\mathbb{F}_p [t])$, 
for all $m \in \mathbb{N}$ satisfying $p \equiv 1 \mod m$, simultaneously. 
However, absent an application, we shall not rehearse the details of such an argument. 
\end{rmrk}

Finally we prove a non-concentration estimate for elements with reducible characteristic polynomial. 

\begin{thm} \label{redcharpolysarerare}
Let $S$ be as in Theorem \ref{sieveupstairs}. There exist $C_1 , C_2 > 0$ such that: 
\begin{center}
$\mu_S ^{(l)} (\lbrace g \in SL_2 (\mathbb{F}_p [t]) : 
\chi_g \text{ is reducible } \rbrace) \leq C_1 e^{-C_2 \sqrt{l/\log(l)}}$. 
\end{center}
\end{thm}

It is very likely that bounds on return probabilities of random walks to other subsets of $SL_2 (\mathbb{F}_p [t])$ may be proved by the same method, and we emphasize that our results are best viewed as sample, rather than an exhaustive list, of the applications of this theory. 

We now turn to our results on expanders. 

\begin{defn}
For $M > 0$, an integer $n>1$ will be called \emph{$M$-admissible} if $n$ has no prime factor less than $M$. 
A polynomial $f \in \mathbb{F}_p [t]$ will be called $M$-admissible 
if the degree of every irreducible factor of $f$ is a $M$-admissible integer. 
\end{defn}

\begin{ex} \label{smallsubfieldsap}
Let $M > 0$. 
\item[(i)] Every prime $> M$ is $M$-admissible. 

\item[(ii)] There is a sequence $(n_i)_i$ of $M$-admissible integers growing linearly in $i$. 
For, given $M$, let $\pi$ be the set of all primes up to $M$. 
Let $n_i = Ni + 1$, where $N = \prod_{P \in \pi} P$. 
It will be significant in the applications in Section \ref{sieve} 
that the set of admissible integers is sufficiently dense. 
\end{ex}

Our main result on constructing expanders is: 

\begin{thm} \label{mainthm}
Let $S \subseteq SL_2 (\mathbb{F}_p [t])$ be a finite symmetric subset, 
generating a non-elementary subgroup. 
Suppose every entry of every element of $S$ has degree at most $D$. 
Let $(f_i)_i \subseteq \mathbb{F}_p [t]$ be a sequence of distinct polynomials. 
Then there exists $M > 0$ (depending on $D$ and $p$) such that, 
if $(f_i)_i$ are $M$-admissible then for $i_0 \in \mathbb{N}$ sufficiently large (depending on $D,p$), 
$(SL_2 (\mathbb{F}_p [t] / (f_i)) , \pi_{f_i}(S))_{i \geq i_0}$ is a two-sided expander family, 
provided one of the following holds:  
\begin{itemize}
\item[(i)] The $f_i$ are irreducible. 

\item[(ii)] The $f_i$ are square-free, every irreducible factor of every $f_i$ has prime degree, 
and no two irreducible factors of any $f_i$ have the same degree. 

\end{itemize}
\end{thm}

We define the concept of a two-sided expander family in Section \ref{expanderssubsection} below. 
Here, and throughout, for $\mathbb{G}$ a linear algebraic group defined over $\mathbb{F}_p$ 
and $f \in \mathbb{F}_p [t]$, 
$\pi_f : \mathbb{G} (\mathbb{F}_p [t]) \rightarrow \mathbb{G} (\mathbb{F}_p [t] / (f))$ shall denote the congruence map. 

One of the keys to the proof of Theorem \ref{mainthm} shall be an analysis of Cayley graphs of large girth. 
For $G$ a finite group generated by a subset $S$, recall that the girth of $(G,S)$ 
is the length of the shortest non-trivial reduced word in $S$ which equals $1$ in $G$, 
or equivalently the length of the shortest non-trivial embedded loop in the Cayley graph $\Cay(G,S)$. 
In the course of our analysis, 
we also obtain the following result, which applies to generating sets which may not be congruence images of a fixed subset in $SL_2 (\mathbb{F}_p [t])$. 

\begin{thm} \label{girthexpanders}
For any $C_1 > 0$ and any $k \in \mathbb{N}$ with $k \geq 2$, 
there exists $C > 0$ (depending on $k,p$ and $C_1$) such that, 
if $(n_i)_i$ a sequence of $C$-admissible positive integers, 
$S_{n_i} \subseteq SL_2 (p^{n_i})$ is symmetric with $\lvert S_{n_i} \rvert = 2k$, 
and $\girth(SL_2 (p^{n_i}) , S_{n_i}) \geq C_1 n_i$, for all $i \in \mathbb{N}$, 
then for $i_0$ sufficiently large (depending on $C_1 , k$), 
$(SL_2 (p^{n_i}) , S_{n_i})_{i \geq i_0}$ is a two-sided expander family. 
\end{thm}

The lower bound on girth is a natural hypothesis: for instance it is satisfied by \emph{generic} subsets of $SL_2 (p^n)$. Indeed large girth is a key component of the proof \cite{BrGrGuTa} that random pairs of generators in $SL_2 (p^n)$ yield expanders. 

\subsection{Expanders} \label{expanderssubsection}

Let $G$ be a finite group. For two functionals $\phi , \psi \in l^2 (G)$, 
the \emph{convolution} $\phi * \psi \in l^2 (G)$ is given by:
\begin{center}
$(\phi * \psi)(g) = \sum_{h \in G} \phi (h) \psi (h^{-1} g)$. 
\end{center}
For $l \in \mathbb{N}$, we define the \emph{convolution power} $\phi^{(l)}$ recursively via: 
\begin{center}
$\phi^{(0)} = \chi_e$; $\phi^{(l+1)} = \phi^{(l)} * \phi$.  
\end{center}
Here $\chi_e$ is the characteristic function of the identity element $e \in G$. 
For $S \subseteq G$ a symmetric subset, 
define  a linear operator $A_S : l^2 (G) \rightarrow l^2 (G)$ (called the \emph{adjacency operator}) by:
\begin{center}
$A_S (f) = (\frac{1}{\lvert S \rvert} \chi_S)  * f$. 
\end{center}
$A_S$ is self-adjoint of operator norm $1$; let its spectrum be: 
\begin{center}
$1 = \lambda_1 \geq \lambda_2 \geq \ldots \lambda_{\lvert G \rvert} \geq -1$
\end{center}
with the eigenvalue $\lambda_1 = 1$ corresponding to the constant functionals on $G$. 
More generally, the $1$-eigenspace of $A_S$ is generated by the indicator functions of 
the right cosets of $\langle S \rangle \leq G$. 
In particular $\lambda_1 > \lambda_2$ iff $S$ generates $G$. 
Let $l_0 ^2 (G) \leq l^2 (G)$ be the space of functions of mean zero on $G$ 
(that is, the orthogonal complement of the constant functions), 
and note that $l_0 ^2 (G)$ is preserved by $A_S$. 
Let $\rho = \max( \lvert \lambda_2 \rvert , \lvert \lambda_{\lvert G \rvert} \rvert )$, 
the norm of the restriction $A_S \mid_{l_0 ^2 (G)}$ 
in the Banach space $B(l_0 ^2 (G))$ of bounded linear operators on $l_0 ^2 (G)$. 

\begin{defn}
For $\epsilon > 0$, the pair $(G,S)$ is a (two-sided) \emph{$\epsilon$-expander} if $\rho \leq 1 - \epsilon$. 
A sequence $(G_n , S_n)_{n \in \mathbb{N}}$ is called an \emph{$\epsilon$-expander family} if $(G_n , S_n)$ 
is an $\epsilon$-expander for every $n \in \mathbb{N}$, or just an \emph{expander family} 
if there exists $\epsilon > 0$ such that $(G_n , S_n)_{n \in \mathbb{N}}$ is an $\epsilon$-expander family. 
\end{defn}

The two-sided version of expansion (also known as \emph{absolute expansion}) 
that we use here is stronger than the one-sided version which will be more familiar to many readers, 
and which is equivalent to the combinatorial notion of expansion defined in terms of the discrete Cheeger constant. 
For the most part, however, the distinction need not concern us, thanks to a recent result of Breuillard, 
Green, Guralnick and Tao \cite{BrGrGuTa}: 

\begin{thm} \label{1vs2sided}
For any $\epsilon > 0 , k \in \mathbb{N}$, there exists $\delta_{\epsilon , k} > 0$ such that, 
if $(G,S)$ is a one-sided $\epsilon$-expander with $\lvert S \rvert = k$, 
then one of the following holds: 
\begin{itemize}
\item[(i)] $(G,S)$ is a two-sided $\delta$-expander; 
\item[(ii)] There exists $H \leq G$ with $\lvert G : H \rvert = 2$ and $S \cap H = \emptyset$. 
\end{itemize}
\end{thm}

We recall some facts about expanders which will be used in what follows. 
Those readers more familiar with the one-sided version of expansion may note that these results about two-sided expanders follow from their one-sided analogues together with Theorem \ref{1vs2sided}. 
Note that condition (ii) of Theorem \ref{1vs2sided} is equivalent to $\Cay(G,S)$ being bipartite. 

\begin{lem} \label{indepgen}
Let $\Gamma$ be a finitely generated group; $(K_n)_n$ be a sequence of finite index subgroups of $\Gamma$; 
$\pi_n : \Gamma \twoheadrightarrow \Gamma / K_n$ be the natural epimorphism. 
Let $S,T \subseteq \Gamma$ be finite symmetric subsets, with $\langle S \rangle = \langle T \rangle = \Gamma$. 
Suppose $\Cay(\Gamma / K_n , \pi_n (T))$ is not bipartite, for all $n \in \mathbb{N}$. 
If $(\Gamma / K_n , \pi_n (S))_n$ is an expander family then so is $(\Gamma / K_n , \pi_n (T))_n$. 
\end{lem}

\begin{lem} \label{biggersubgrp}
Let $\Gamma$; $(K_n)_n$; $\pi_n$ be as in Lemma \ref{indepgen} 
and let $H \leq \Gamma$ be a finitely generated subgroup. 
Suppose $\pi_n (H) = \Gamma / K_n$ for all $n \in \mathbb{N}$. 
Let $S \subseteq \Gamma$, $T \subseteq H$ be finite symmetric subsets, 
with $\langle S \rangle = \Gamma$, $\langle T \rangle = H$. 
If $(\Gamma / K_n , \pi_n (T))_n$ is an expander family, 
and $\Cay(\Gamma / K_n , \pi_n (S))$ is not bipartite, 
then $(\Gamma / K_n , \pi_n (S))_n$ is an expander family. 
\end{lem}

In all cases in which we shall use these results, 
the finite groups concerned shall have no subgroups of index two, 
so the associated Cayley graphs shall never be bipartite. 

The expander property for the pair $(G,S)$ provides a logarithmic bound 
for the mixing times of random walks on $G$, 
given by a probability measure supported on $S$. 
This will be useful in the applications in Section \ref{sieve}. 

\begin{lem} \label{mixingbound}
For any $l \in \mathbb{N}$; $g , h \in G$, 
\begin{center}
$\lvert \langle A_S ^l \chi_g , \chi_h \rangle - \frac{1}{\lvert G \rvert} \rvert \leq  \rho^l$. 
\end{center}
\end{lem}

\subsection{The Bourgain-Gamburd Machine}

In \cite{BoGa1} Bourgain and Gamburd developed a new method for constructing expanders, 
exploiting results from additive combinatorics. Much has been written about the possible formulations of the Bourgain-Gamburd Machine \cite{BreSurvey}, \cite{TaoBook}, so we shall not rehearse the details of the proof, and give only a rough sketch of the most salient points. 

The Machine tells us that the expansion of a pair $(G,S)$ may be guaranteed by three hypotheses. 
The first is that $G$ should be highly \emph{quasirandom}, meaning that $G$ has no 
small-dimensional non-trivial complex representations. 
There are good classical estimates of quasirandomness for many familiar families of finite groups, 
including finite simple groups of Lie type. 
The combinatorics of quasirandom groups was studied by Gowers \cite{Gowers}, 
who coined the term, but its connection with expansion was first noted by Sarnak and Xue \cite{SarXue}. 
Suppose $A_S$ has an eigenvalue $\lambda$ of modulus close to $1$, so that the expansion is weak. 
The eigenspace of $\lambda$ is a subrepresentation of the right-regular representation of $G$, 
so by quasirandomness has large dimension. This places a lower bound on $\tr(A_S ^{2l})$. 

Proving expansion therefore reduces to showing that $\tr(A_S ^{2l})$ decays quickly. 
Sufficient conditions for such decay come from the non-commutative 
Balog-Szemer\'{e}di-Gowers Theorem, due to Tao \cite{TaoBSG}, which tells us that 
if decay fails, the measure $\mu_S ^{(2l)}$ must concentrate somewhat 
on a small \emph{approximate subgroup} $A$ of $G$ 
(that is, a symmetric subset containing $1$ such that $AA$ is covered by a small number of translates of $A$). 
Some papers utilising the Bourgain-Gamburd Machine have tackled the problem of excluding this possibility head-on, 
but we can reduce the problem still further if $G$ satisfies a \emph{product theorem}. 
All finite groups have some obvious approximate subgroups: if $A$ is already almost the whole of $G$, 
or is almost a proper subgroup of $G$, then $A$ will not grow much under multiplication with itself. 
A product theorem says roughly that these are the only possibilities. 
Through much work by many authors in the past decade, 
product theorems are now known for many families of groups, 
including finite (quasi)simple groups of Lie type of fixed rank \cite{BrGrTa}, \cite{PySz}. 

Expansion is thereby reduced to showing that $\mu_S ^{(2l)}$ escapes quickly from proper subgroups of $G$. 
It should be noted that this is the only point at which the set $S$ enters the argument. 
We therefore usually expect some information about the geometry of $S$ to be crucially involved in 
proving escape from subgroups. 

The general form of the Bourgain-Gamburd Machine admits many, related but distinct, formulations. 
The version which it shall be most convenient for us to consider shall be the following. 

\begin{thm} \label{BGMachine}
Let $G$ be a finite group; $S \subseteq G$ a symmetric generating set. Suppose:
\begin{itemize}
\item[(i)](Quasirandomness) There exists $\alpha > 0$ such that, 
for some finite groups $G_1 , \ldots , G_n$, $G=\prod_{i=1} ^n G_i$ and for every $i$, 
for any non-trivial representation $\rho : G_i \rightarrow GL_d (\mathbb{C})$, 
\begin{center}
$d \geq \lvert G_i \rvert^{\alpha}$. 
\end{center}
\item[(ii)](Product theorem) For all $\delta > 0$ there exists $\beta (\delta) > 0$ such that 
for any symmetric $A \subseteq G$ such that:
\begin{center}
$\lvert A \rvert < \lvert G \rvert^{1-\delta}$ and $\mu_A (gH) < [G:H]^{-\delta} \lvert G \rvert^{\beta}$
\end{center}
for any $g \in G$ and $H \lneq G$, then $\lvert A A A \rvert \gg \lvert A \rvert^{1+\beta}$. 

\item[(iii)](Escape from subgroups) There exists $\gamma > 0$ and $C_0 > 0$ such that 
for some $l \leq C_0 \log \lvert G \rvert$ and every $H \lneq G$,
\begin{center}
$\mu_S ^{(2l)} (H) \leq \lvert G : H \rvert^{-\gamma}$. 
\end{center}
\end{itemize}
Then there exists $\epsilon (\alpha,\beta(\cdot),\gamma,C_0,\lvert S \rvert) > 0$ 
such that $(G,S)$ is a two-sided $\epsilon$-expander. 
\end{thm}

A proof of this result is contained in Sections 3.1 and 5 of \cite{Varju}. 
See also Section 3 of \cite{BreSurvey} for a self-contained and accessible account of a related 
(though non-equivalent) version of the Machine. 

\begin{rmrk} \label{weakernonconrmrk}
For $H \leq G$, and $\phi \in l^2 (G)$, define $\overline{\phi} \in l^2 (G/H)$ by:
\begin{center}
$\overline{\phi} (gH) = \sum_{h \in H} \phi (gh)$. 
\end{center}
Then since $S$ is symmetric, for $l \in \mathbb{N}$, 
\begin{center}
$\mu_S ^{(2l)} (H) = \lVert \overline{\mu_S ^{(l)}} \rVert_2 ^2$. 
\end{center}
Now define $\overline{A_S} : l^2 (G/H) \rightarrow l^2 (G/H)$ by: 
\begin{center}
$\overline{A_S} (F) (gH) = \frac{1}{\lvert S \rvert} \sum_{s \in S} F (sgH)$. 
\end{center}
Then $\overline{A_S}$ is a linear operator; it is a contraction 
(being the adjacency operator on the Schreier graph of $(G/H,S)$) 
and satisfies, for $\phi \in l^2 (G)$, 
\begin{center}
$\overline{A_S} (\overline{\phi}) = \overline{\mu_S \ast \phi}$. 
\end{center}
It follows that $\mu_S ^{(2l)} (H)$ is a decreasing function of $l$. 
Hypothesis (iii) of Theorem \ref{BGMachine} therefore follows from an apparently weaker variant, 
in which our $l \leq C_0 \log \lvert G \rvert$ is permitted to depend on the subgroup $H$. 
\end{rmrk}

\subsection{Further Questions and Structure of the Paper}

It is natural to ask whether the admissibility hypothesis in Theorems \ref{mainthm} and \ref{girthexpanders} may be weakened. However there are some significant obstacles to doing so. 
For instance it is clear that Theorem \ref{girthexpanders} does not remain true for arbitrary sequences $(n_i)_i$: 

\begin{ex} \label{girthnotconnected}
Let $n_i = 2^i$. Then we may identify $\mathbb{F}_{p^{n_i}}$ with a proper subfield of $\mathbb{F}_{p^{n_{i+1}}}$, 
and hence embed $SL_2 (p^{n_i}) \hookrightarrow SL_2 (p^{n_{i+1}})$. 
For $i$ even, let $S_{n_i}$ be a generating set for $SL_2 (p^{n_i})$ 
satisfying $\girth(SL_2 (p^{n_i}),S_{n_i}) \gg n_i$. 
For $i$ odd, let $S_{n_i} = S_{n_{i-1}}$, so that $\langle S_{n_i} \rangle \lneq SL_2 (p^{n_i})$. 
Then for every $i$, $\girth(SL_2 (p^{n_i}),S_{n_i}) \gg n_i$, 
but $\lbrace (SL_2 (p^{n_i}),S_{n_i}) \rbrace_{i \geq j}$ is not an expander family for any $j$. 
\end{ex}

So the presence of large subfield subgroups presents a genuine obstruction to expansion of subsets. 
It should be noted however that Example \ref{girthnotconnected} exhibits an obstruction to expansion 
which is \emph{qualitative}, rather than \emph{quantitative} in nature. 
That is to say, expansion in $(SL_2 (p^{n_i}),S_{n_i})$ 
fails simply by virtue of the fact that $\langle S_{n_i} \rangle \neq SL_2 (p^{n_i})$ for infinitely many $i$. 
This leads to the question of whether this is the \emph{only} obstruction to expansion in these groups. Specifically: 

\begin{qu} \label{girthgenexpex}
Let $S_n \subseteq SL_2 (p^n)$ with $\girth(SL_2 (p^n),S_n) \gg n$. Does there exist $\epsilon > 0$ 
such that $(\langle S_n \rangle,S_n)$ is an $\epsilon$-expander for all $n$ sufficiently large? 
\end{qu}

\begin{qu} \label{conggenexpex}
Let $S \subseteq SL_2 (\mathbb{F}_p [t])$ be a finite symmetric set generating a non-elementary subgroup. 
Let $(f_i)_i \subseteq \mathbb{F}_p [t]$ be a sequence of distinct irreducible polynomials. 
Does there exist $\epsilon > 0$ such that $(\langle \pi_{f_i} (S) \rangle,\pi_{f_i} (S))$ 
is an $\epsilon$-expander for all $i$ sufficiently large? 
\end{qu}

A second way in which Theorem \ref{mainthm} (ii) might be extended would be relax the assumption that 
no two irreducible factors of $f_i$ have the same degree. 
As a model case, let $f,g \in \mathbb{F}_p [t]$ be distinct irreducibles of degree $n$, 
and consider the group $SL_2 (\mathbb{F}_p [t]/(f \cdot g))$. 
By the Chinese Remainder Theorem, this may be identified with $SL_2 (p^n) \times SL_2 (p^n)$. 
A potential obstruction to expansion in this group comes from proper subdirect products of 
$SL_2 (p^n) \times SL_2 (p^n)$, which arise as the graphs of automorphisms of $SL_2 (p^n)$. 
It remains an open question how to demonstrate escape from such subgroups, 
as would be required for hypothesis (iii) of Theorem \ref{BGMachine}. 

The primality assumption in Theorem \ref{mainthm} (ii) comes from hypothesis (ii) of Theorem \ref{BGMachine}, 
which in our setting is satisfied by results of Varj\'{u} \cite{Varju}. 
The applicability of these results shall be discussed in more detail in Section \ref{ExpanderSection}.  
Roughly speaking though, for the product theorem to apply to reductions modulo polynomials with unboundedly many irreducible factors (so that the corresponding congruence quotients decompose as products with unboundedly many quasisimple factors), the subgroup structure of the quasisimple factors must be highly restricted. 
It seems plausible that a generalisation of Varj\'{u}'s product theorem which relaxes these restrictions may be discovered, and the primality assumption thereby removed. 

An expansion result for reductions modulo arbitrary square-free polynomials seems even further out of reach. 
For then the decompositions of the congruence quotients into products of quasisimple groups 
contain unboundedly many isomorphic factors, so Varj\'{u}'s product theorem fails even more dramatically. 
It may be that the fastest route to a result on expansion in this general setting is to tackle the question of concentration in approximate subgroups directly. 

Even an expansion result in the case of two irreducible factors of the same degree would have useful 
consequences for sieving in $SL_2 (\mathbb{F}_p [t])$. For in the presence of such a result 
(and the relevant strengthening of the product theorem indicated above) 
we could substitute the group sieve of Lubotzky-Meiri for Proposition \ref{bigsieve} in the proofs of 
Theorems \ref{squaresarerare} and \ref{redcharpolysarerare}, 
thereby improving the upper bounds in those two results from $e^{-C \sqrt{l/\log(l)}}$ to $e^{-C l}$. 

The paper is structured as follows: 
in Section \ref{ExpanderSection} we prove Theorems \ref{mainthm} and \ref{girthexpanders}. 
Specifically, Section \ref{freereduxsection} shall deal with hypotheses (i) and (ii) of Theorem \ref{BGMachine} 
and further reduce Theorem \ref{mainthm} to the case of non-abelian free subgroups. 
We then turn to hypothesis (iii) of Theorem \ref{BGMachine}. 
In Section \ref{IrrednonconSection} it is verified for Cayley graphs of $SL_2 (p^n)$ 
with large girth under the admissibility hypothesis. 
This yields Theorem \ref{mainthm} (i) and Theorem \ref{girthexpanders}. 
The generalisation of this argument required for Theorem \ref{mainthm} (ii) 
is explained in Section \ref{GennonconSection}. 

We discuss the applications to random walks in $SL_2 (\mathbb{F}_p [t])$ in Section \ref{sieve}. 
In Section \ref{TwoSievesSection} we explain in general terms how non-concentration results 
in infinite groups can be obtained using expansion results on finite quotients. 
Theorems \ref{sieveupstairs}, \ref{squaresarerare} and \ref{redcharpolysarerare} 
are proved in the subsequent three Sections. 

I would like to thank my supervisor, Marc Lackenby, 
for the abundance of support and sound advice he has given me in preparing this work, 
and his ongoing enthusiasm for my research in general. 
I am also grateful to Emmanuel Breuillard, Alireza Salehi-Golsefidy and Peter Varj\'{u} for enlightening conversations. 
Finally, I would like to thank EPSRC for providing financial support during the undertaking of this work. 

\section{Constructing the Expanders} \label{ExpanderSection}

As a notational convenience, for $n \in \mathbb{N}$ we set $Q_n = SL_2 (p^n)$. 

\subsection{Reduction to Escape for Free Generators} \label{freereduxsection}

In this section we reduce the proof of Theorem \ref{mainthm} to the following Proposition: 

\begin{propn} \label{nonconpropn}
Let $T \subseteq SL_2 (\mathbb{F}_p [t])$ be the symmetric closure of a finite subset, 
freely generating a non-abelian free subgroup. 
Suppose every entry of every element of $T$ has degree at most $\tilde{D}$. 
Then there exists $C,M,\gamma > 0$ (depending on $\tilde{D}$, $\lvert T \rvert$ and $p$) 
such that the following holds. 
Let $f \in \mathbb{F}_p[t]$ be an $M$-admissible square-free polynomial 
with no two irreducible factors having the same degree. 
Then for every $H \lneq G = SL_2 (\mathbb{F}_p [t]/(f))$, 
there exists $l \leq C \log \lvert G \rvert$ such that: 
\begin{center}
$\mu_{T} ^{(2l)} (H) \leq \lvert G : H \rvert^{-\gamma}$. 
\end{center}
\end{propn}

The reduction shall be via Theorem \ref{BGMachine}. 
We reference known results which cover hypotheses (i) and (ii) of Theorem \ref{BGMachine}. 
We then use the general results about expanders from Section \ref{expanderssubsection} 
to reduce the question of expansion for arbitrary sets $S$ as in the Statement of Theorem \ref{mainthm} 
to expansion for finite sets $T \subseteq \langle S \rangle$ freely generating $\langle T \rangle$. 
This shall be via a Tits alternative. 

The quasirandomness condition in our setting is classical (see for instance \cite{LandSeitz}):

\begin{thm} \label{quasirandomness}
There is an absolute constant $C>0$ such that every non-trivial complex representation of $Q_n$ 
has dimension at least $C p^n$. 
\end{thm}

Let $f$ be as in Proposition \ref{nonconpropn} 
and let $p_1 , \ldots , p_N$ be the irreducible factors of $f$, 
of degrees $n_1 , \ldots , n_N$ respectively. 
It follows from the Chinese Remainder Theorem that:
\\
\begin{lem} \label{CRT}
The natural map:
\begin{center}
$(\prod_{j=1} ^N \pi_{p_j}) : SL_2 (\mathbb{F}_p [t] / (f)) \rightarrow \prod_{j=1} ^N Q_{n_j}$ 
\end{center}
is an isomorphism. 
\end{lem}

We turn next to the product theorem. In the setting of Theorem \ref{mainthm} (i), this is due to Dinai \cite{Dinai}. 
For Theorem \ref{mainthm} (ii) we use Proposition 14 of \cite{Varju}, which we quote in full: 

\begin{thm} \label{uniformgrowth}
For all $\delta > 0$, $L \in \mathbb{N}$ and $\beta : \mathbb{R}^+ \rightarrow \mathbb{R}^+$ 
there exists $\beta' _{L,\beta} (\delta) > 0$ such that the following holds. 
Let $G$ be a finite group, $G_1 , \ldots , G_N$ be finite groups such that $G \cong G_1 \times \cdots \times G_N$. Suppose: 
\begin{itemize}
\item[(i)] For any finite group $F$, $\lvert \lbrace i \in \lbrace 1 , \ldots , N \rbrace : G_i \cong F \rbrace \rvert \leq L$. 
\item[(ii)] For $1 \leq i \leq N$, $G_i$ is quasisimple and $\lvert Z(G_i) \rvert \leq L$. 
\item[(iii)] For $1 \leq i \leq N$, any non-trivial complex representation of $G_i$ has dimension at least $\lvert G_i \rvert^{\frac{1}{L}}$. 
\item[(iv)] For $1 \leq i \leq N$ and for some $m < L$, there are classes $\mathcal{H}_0 , \mathcal{H}_1 , \ldots , \mathcal{H}_m$ of subgroups of $G_i$ satisfying: 
\begin{itemize}
\item[(a)] $\mathcal{H}_0 = \lbrace Z (G_i) \rbrace$. 
\item[(b)] Each $\mathcal{H}_j$ is closed under conjugation in $G_i$. 
\item[(c)] For each $H < G_i$ there is $1 \leq j \leq m$ and $H^{\sharp} \in \mathcal{H}_j$ such that 
$\lvert H:H \cap H^{\sharp} \rvert \leq L$. 
\item[(d)] For $1 \leq j \leq m$ and for each $H_1 , H_2 \in \mathcal{H}_j$ with $H_1 \neq H_2$, there exists $j' < j$ and $H^{\sharp} \in \mathcal{H}_{j'}$ such that $\lvert H_1 \cap H_2 : H_1 \cap H_2 \cap H^{\sharp} \rvert \leq L$.
\end{itemize}
\end{itemize}
If $G_1 , \ldots , G_N$ satisfy hypothesis (ii) of Theorem \ref{BGMachine} with respect to the function $\beta$, 
then $G$ satisfies hypothesis (ii) of Theorem \ref{BGMachine} with respect to the function $\beta' _{L,\beta}$. 
\end{thm}

We check that this result applies to $G = SL_2 (\mathbb{F}_p [t] / (f_i))$, for $f_i$ as in Theorem \ref{mainthm} (ii). The decomposition as a product is given by Lemma \ref{CRT}. (i) follows from the assumption that no two irreducible factors of $f_i$ have the same degree. (ii) is well-known for $G_i = Q_{n_i}$. 
(iii) is Theorem \ref{quasirandomness}. For (iv), we recall the classification of subgroups of $Q_n$ (see for instance \cite{OHKing}). 

\begin{propn} \label{subgroupclassification}
For $\mathbb{F}_q$ the finite field of order $q$ and characteristic $p \geq 3$, 
any proper subgroup $H$ of $SL_2 (\mathbb{F}_q)$ satisfies one of the following: 
\begin{center}
\begin{itemize}
\item[(i)] $H$ fixes a point in the projective line $\mathbb{F}_{q^2} \mathbb{P}^1$ over the quadratic extension $\mathbb{F}_{q^2}$ of $\mathbb{F}_q$. In particular $H$ is metabelian. \\
\item[(ii)] $H \leq S_5$. \\
\item[(iii)] $H$ is conjugate in $SL_2 (\mathbb{F}_q)$ to a subgroup of 
$SL_2 (\mathbb{E})$ for some proper subfield $\mathbb{E}$ of $\mathbb{F}_q$.
\end{itemize}
\end{center}
\end{propn}

Define $\mathcal{H}_1$ to be the set of stabilisers in $Q_n$ of pairs of distinct points in $\mathbb{F}_{p^{2n}} \mathbb{P}^1$, and $\mathcal{H}_2$ to be the set of stabilisers in $Q_n$ of points in $\mathbb{F}_{p^{2n}} \mathbb{P}^1$. We check that the conditions of Theorem \ref{uniformgrowth} (iv) are satisfied by $\mathcal{H}_0 = \lbrace Z (Q_n) \rbrace,\mathcal{H}_1,\mathcal{H}_2$, in the case for which $n$ is prime. (a), (b) are obvious, and (c) is immediate from Proposition \ref{subgroupclassification}, since by primality of n, the only proper subfield subgroups of $Q_n$ are the conjugates of $Q_1$, which are of bounded size. (d) is a consequence of the following elementary fact from linear algebra: 

\begin{lem}
Suppose $g \in Q_n$ has at least three distinct fixed points in $\mathbb{F}_{q^2} \mathbb{P}^1$. Then $g \in Z(Q_n)$. 
\end{lem}

Now let $S$ be as in the statement of Theorem \ref{mainthm}. 
We produce a pair of words in $S$ freely generating a non-abelian free subgroup. 
In the classical Tits alternative, the lengths of our free generators as words in $S$, 
and hence the degrees of their entries, depend on $S$ and not just on $D$. 
However, we can obtain a bound depending only on $D$ by utilising the following result of Breuillard: 

\begin{thm}[Uniform Tits Alternative \cite{UnifTits}] \label{uniformtits}
 For every $d \geq 2$, there exists $N(d) > 0$ such that, 
 for any field $K$, and $S \subseteq SL_d (K)$ finite symmetric, 
 either $\langle S \rangle$ is virtually soluble or the ball $B_S (N(d))$ of radius $N(d)$ 
 in the word metric contains two elements which freely generate a non-abelian free subgroup of $SL_d (K)$. 
\end{thm}

\begin{proof}[Proof of Theorem \ref{mainthm}]
Let $N=N(2)$ be as in Theorem \ref{uniformtits} and let \linebreak $x , y \in B_S (N)$ freely generate a non-abelian free group. 
Every entry of every member of $T = \lbrace x^{\pm 1} , y^{\pm 1} \rbrace$ 
is expressible as a sum of monomials of degree at most $N$ in the entries of the elements of $S$, 
hence has degree at most $DN$. We now apply Theorem \ref{BGMachine} to $(SL_2 (\mathbb{F}_p[t]/(f_i))),\pi_{f_i}(T))$. 
\\ \\
We verify the conditions of Theorem \ref{BGMachine}. 
Hypothesis (i) is immediate from Theorem \ref{quasirandomness} and Lemma \ref{CRT}. 
Hypothesis (ii) follows from \cite{Dinai} and Theorem \ref{uniformgrowth}. 
Hypothesis (iii) follows from Proposition \ref{nonconpropn}, applied with $f=f_i$ and $\tilde{D}=DN$, and Remark \ref{weakernonconrmrk}. 
\\ \\
We conclude that $(SL_2 (\mathbb{F}_p[t]/(f_i)),\pi_{f_i}(T))_i$ is an expander family. 
By Lemma \ref{biggersubgrp}, $(SL_2 (\mathbb{F}_p[t]/(f_i)),\pi_{f_i}(B_S (N)))_i$ is an expander family. 
The required result follows from Lemma \ref{indepgen}, since $\langle S \rangle = \langle B_S (N) \rangle$. 
\end{proof}

\begin{rmrk}
The constants $M$ and $i_0$ in the statement of Theorem \ref{mainthm} could in principle be computed, 
by keeping track of the bounds arising in the proof of Proposition \ref{nonconpropn} below. 
They shall involve both the constant $N$ from the statement of the Uniform Tits Alternative 
and the known spectral radius $\sqrt{8/9}$ 
for the simple random walk on $\lbrace x^{\pm 1} ,y^{\pm 1} \rbrace$ in $F(x,y)$. 
Moreover, the proof of the Uniform Tits Alternative is effective, 
so $N$ could in principle be computed (though to our knowledge this has not been done). 
Such a computation would yield an explicit description of the degrees of reductions which would give rise to families of expanders, in terms only of the degrees of the entries of elements of $S$, $\lvert S \rvert$ and $p$. 
\end{rmrk}

\subsection{Escape from Subgroups: The Irreducible Case} \label{IrrednonconSection}

In this section we warm up to the proof of Proposition \ref{nonconpropn} 
by examining the case for which the polynomials $f_i$ are irreducible, 
so that $SL_2 (\mathbb{F}_p [t]/(f_i)) = Q_{\deg(f_i)}$. 
The proof of this case shall contain all the key ideas of the general case 
(to be discussed in the following section) but is technically simpler. 
Indeed, more generally we shall prove: 

\begin{propn} \label{nonconsubgrps1}
For any $C_1 > 0$ and any $k \in \mathbb{N}$ with $k \geq 2$, 
there exists $C_2,C_3, \gamma>0$ 
(depending on $C_1 , p , \lvert S \rvert$) such that, 
if $n$ is a $C_2$-admissible positive integer, 
$S_n \subseteq Q_n$ is symmetric with $\lvert S_n \rvert = 2k$, 
and  \linebreak$\girth(Q_n , S_n) \geq C_1 n$, 
then for $n$ sufficiently large and for all $H_n \lneq Q_n$, there exists $l \leq C_3 \log \lvert Q_n \rvert$ such that: 
\begin{center}
$\mu_{S_n} ^{(2l)} (H_n) \leq \lvert Q_n \rvert^{-\gamma}$. 
\end{center}
\end{propn}

The relevant case of Proposition \ref{nonconpropn} follows immediately from 
Proposition \ref{nonconsubgrps1} and the following Lemma: 

\begin{lem} \label{girthlem}
let $T$ be as in Proposition \ref{nonconpropn} and $f \in \mathbb{F}_p [t]$ be of degree $n$. Then: 
\begin{center}
$\girth(SL_2 (\mathbb{F}_p [t]/(f)) , \pi_f (T)) \geq n / \tilde{D}$. 
\end{center}
\end{lem}

\begin{proof}
Let $w$ be a non-trivial reduced word in $T$ of length $l$. 
Every entry of every element of $T$ has degree at most $\tilde{D}$, 
so every entry of $w$ has degree at most $\tilde{D} l$. Now suppose $\pi_f (w) = 1$, 
so that $w \in (I_2 + f \cdot \mathbb{M}_2 (\mathbb{F}_p [t])) \setminus \lbrace I_2 \rbrace$. 
Then at least one entry of $w$ has degree at least $n$, so $l \geq n / \tilde{D}$, as required. 
\end{proof}

Given the discussion in Section \ref{freereduxsection}, 
Proposition \ref{nonconsubgrps1} also immediately implies Theorem \ref{girthexpanders}. 

\begin{proof}[Proof of Theorem \ref{girthexpanders}]
As in the proof of Theorem \ref{mainthm}, we apply Theorem \ref{BGMachine}. 
Hypothesis (i) is Theorem \ref{quasirandomness}; hypothesis (ii) is \cite{Dinai} 
and hypothesis (iii) follows from Proposition \ref{nonconsubgrps1} and Remark \ref{weakernonconrmrk}. 
\end{proof}

We now turn to the proof of Proposition \ref{nonconsubgrps1}. 
Once again we exploit the classification of subgroups of $Q_n$. 

Informally, in all cases, the girth hypothesis and Kesten's Theorem will 
reduce the problem of bounding $\mu_{S_n} ^{(2l)} (H_n)$ 
to providing an upper bound for $\lvert H_n \cap B_{S_n} (2l) \rvert$. 
For $H_n \leq S_4$ or $A_5$ this is immediate. 
The admissibility hypothesis on $n$ will guarantee that 
any proper subfield subgroup is too small to fill $\lvert B_{S_n} (2l) \rvert$. 
Non-concentration in metabelian subgroups will be achieved by the same combinatorial argument 
as was used for the corresponding case in \cite{BoGa1}: 
a metabelian group satisfies a short group law, 
so that if $\lvert H_n \cap B_{S_n} (2l) \rvert$ is large, 
there will be many short relations between the elements of $S_n$. 
However the girth hypothesis guarantees that this will not happen. 

First we recall: 

\begin{thm}[Kesten]
Let $X$ be a finite set. Then there exists  \linebreak$C_4 (\lvert X \rvert) > 0$ such that $\mu_X \in l^2 (F(X))$ satisfies: 
\begin{center}
$\mu_X ^{(2l)} (g) \ll_{\lvert X \rvert} e^{- C_4 l}$ 
\end{center}
for all $g \in F(X)$. 
\end{thm}

The technicalities of non-concentration in subgroups are contained in the following general Lemma. 

\begin{lem} \label{noncontypes}
Let $G$ be a finite group, $S \subseteq G$ symmetric with $\lvert S \rvert = 2k$ 
and let $C_1 > 0$ be such that $\girth(G,S) > C_1 \log \lvert G \rvert$. 
Let $C_4 (k) > 0$ be the constant from Kesten's Theorem. 
Let $H \lneq G$. Let $\gamma > 0$ and let $C_5 \in (0,C_1)$. 
Suppose $\lvert G \rvert$ is sufficiently large. 
\begin{itemize}
\item[(i)] Suppose $H$ is metabelian. 
Suppose $C_5 \log \lvert G \rvert \leq 2l \leq \frac{C_1}{32} \log \lvert G \rvert$. 
Suppose $\gamma < C_4 C_5 / 2$. Then $\mu_S ^{(2l)} (H) \leq \lvert G \rvert^{-\gamma}$. 

\item[(ii)] Let $C_6>0$ and suppose $\lvert H \rvert \leq C_6 \lvert G \rvert^{\frac{1}{C_2}}$
Suppose  \linebreak$C_5 \log \lvert G \rvert \leq 2l \leq C_1 \log \lvert G \rvert$. 
Suppose $\gamma + 1/C_2 < C_4 C_5 / 2$. Then $\mu_S ^{(2l)} (H) \leq \lvert G \rvert^{-\gamma}$. 
\end{itemize}
\end{lem}

\begin{proof}
\begin{itemize}
\item[(i)] 
Define a homomorphism $\theta : F \rightarrow G$ 
from a non-abelian free group $F$ on free basis $X$, 
such that $\theta$ maps $X \cup X^{-1}$ bijectively onto $S$. 
Then $\theta$ maps $B_X(32l)$ bijectively onto $B_S (32l)$. 
\\ \\
Consider $Y = B_X(2l) \cap \theta^{-1} (H)$. 
Then for any $a,b,c,d \in Y$,  \linebreak$\theta([[a,b],[c,d]])=1$ (since $H$ is metabelian) 
so $[[a,b],[c,d]]=1$ (since $[[a,b],[c,d]] \in B_X(32l)$). 
\\ \\
Recall that the centraliser of every non-trivial element of a free group is cyclic. 
Hence there exists $x \in F$ such that for all $a,b \in Y$, 
\begin{equation} \label{freegrpcomms}
[a,b] \in Z:= \langle x \rangle \cap B_X (8l) 
\end{equation}
so that $\lvert Z \rvert \leq 16l+1$. 
Now for $a \in Y$ and $z \in Z$, define: 
\begin{center}
$W_{a,z} = \lbrace b \in Y : [a,b]=z \rbrace$. 
\end{center}
Then $W_{a,z}$ is contained in a single coset of the centraliser of $a$, 
and in $B_X(2l)$, so that $\lvert W_{a,z} \rvert \leq 4l+1$. 
Fix $a \in Y$. By (\ref{freegrpcomms}), 
\begin{center}
$Y \subseteq \bigcup_{z \in Z} W_{a,z}$. 
\end{center}
We conclude that: 
\begin{center}
$\lvert H \cap B_S (2l) \rvert \leq \lvert Y \rvert \leq (16l+1)(4l+1)$. 
\end{center}
By Kesten's Theorem and the girth hypothesis, 
\begin{center}
$\mu^{(2l)} _S (H) \ll_k e^{-C_4 l} \lvert H \cap B_S (2l) \rvert \ll l^2 e^{-C_4 l}$, 
\end{center}
so decays exponentially fast. 

\item[(ii)] 
Suppose (for a contradiction) that for some $C_5 \log \lvert G \rvert \leq 2l \leq C_1 \log \lvert G \rvert$, 
\begin{center}
$\lvert G \rvert^{-\gamma} < \mu_S ^{(2l)} (H) 
\ll_k e^{-C_4 l} \lvert H \rvert \leq C_6 e^{-C_4 l} \lvert G \rvert^{\frac{1}{C_2}}$. 
\end{center}
(the second inequality being by Kesten's Theorem and the girth hypothesis). Hence: 
\begin{center}
$\lvert G \rvert^{\frac{1}{C_2} + \gamma} \gg_{k,C_6} e^{C_4 l}$. 
\end{center}
But $e^{C_4 l} \geq \lvert G \rvert^{\frac{C_4 C_5}{2}}$, 
so choosing $C_2$ sufficiently large and $\gamma$ sufficiently small, 
we have the required contradiction. 
\end{itemize}
\end{proof}

\begin{proof}[Proof of Proposition \ref{nonconsubgrps1}]
Suppose $C_5 n \leq 2l \leq \frac{C_1}{32} n$, for some $C_5 \in (0,\frac{C_1}{32})$. 
We consider each case of Proposition \ref{subgroupclassification} separately: 
\begin{itemize}
\item[(i)] Suppose $H_n$ is metabelian. 
Choosing $\gamma \in (0, C_4 C_5 / 2)$, 
the required result follows from Lemma \ref{noncontypes} (i). 
\item[(ii)] If $H_n \leq S_5$, then $\lvert H_n \rvert \leq 120$, so $\mu^{(2l)} _{S_n} (H_n) \ll_k e^{-C_4 l}$ for any $2l \leq C_1 n$, by Kesten's Theorem and the girth hypothesis. \\
\item[(iii)] Suppose that there exists a proper subfield $\mathbb{E} < \mathbb{F}_{p^n}$ 
such that $H_n$ is contained in (some conjugate of) $SL_2 (\mathbb{E})$. 
Recall that there exists $m \mid n$ such that $\mathbb{E} = \mathbb{F}_{p^m}$. 
By the admissibility hypothesis, $m \leq n / C_2$ so:
\begin{center}
$\lvert H_n \rvert \leq \lvert Q_m \rvert \leq p^{3m} 
\leq (p^{3n})^{\frac{1}{C_2}} \ll_p \lvert Q_n \rvert^{\frac{1}{C_2}}$. 
\end{center}
Choosing $\gamma$ sufficiently small and $C_2$ sufficiently large, 
we may suppose $\gamma + 1/C_2 < C_4 C_5 / 2$, and the result follows from Lemma \ref{noncontypes} (ii). 
\end{itemize}
\end{proof}

\subsection{Escape from Subgroups: The General Case} \label{GennonconSection}

In this Section we complete the proof of Proposition \ref{nonconpropn}. 
The proof shall be very similar in spirit to that of the special case discussed in Section \ref{IrrednonconSection}: 
recall that there, Proposition \ref{subgroupclassification} 
guaranteed that every proper subgroup of $SL_2 (\mathbb{F}_p [t]/(f))$
was either metabelian (Case (i)) or small (Cases (ii) and (iii)), 
so fell within reach of Lemma \ref{noncontypes}. 
Something similar is true in general, but to apply Lemma \ref{noncontypes} 
we first need to use the product decomposition of $SL_2 (\mathbb{F}_p [t]/(f))$ from Lemma \ref{CRT}, 
and project down to either the factors on which the image of our proper subgroup is metabelian, 
or those on which it is small, depending on which make up the larger part of the product. 

Recall the notation of Section \ref{freereduxsection}: 
$f \in \mathbb{F}_p[t]$ is an $M$-admissible square-free polynomial 
with no two irreducible factors having the same degree. 
 \linebreak$G = SL_2 (\mathbb{F}_p [t]/(f))$ and $H \lneq G$. 
Let $p_1 , \ldots , p_N$ be the irreducible factors of $f$, 
of degrees $n_1 , \ldots , n_N$ respectively. Recall (Lemma \ref{CRT}) that: 
\begin{center}
$(\prod_{j=1} ^N \pi_{p_j}) : SL_2 (\mathbb{F}_p [t] / (f)) \rightarrow \prod_{j=1} ^N Q_{n_j}$ 
\end{center}
is an isomorphism. 

\begin{coroll} \label{subgrpsnotontofactors}
$\pi_{p_j} (H) \lneq Q_{n_j}$ for some $1 \leq j \leq N$. 
\end{coroll}

\begin{proof}
We proceed by induction on $N$ (the case $N=1$ being trivial). 
Suppose (for a contradiction) that the projections $\pi_{p_j}$ of $H$ to $Q_{n_j}$ are all surjective. 
Denote $F = \prod_{j=1} ^{N-1} Q_{n_j}$, so that by Lemma \ref{CRT}, $G \cong F \times Q_{n_N}$. Define: 
\begin{center}
$K_1 = \lbrace g \in F : (g,e) \in H \rbrace$, $K_2 = \lbrace g \in Q_{n_N} : (e,g) \in H \rbrace$. 
\end{center}
By induction the projections of $H$ to $F$ and $Q_{n_N}$ are surjective. 
By Goursat's Lemma, $K_1 \vartriangleleft F$, $K_2 \vartriangleleft Q_{n_N}$ and $F/K_1 \cong Q_{n_N} / K_2$. 
\\ \\
If $K_2 \neq Q_{n_N}$ then $F$ has $PSL_2 (p^{n_N})$ as a composition factor. 
But this is not the case, as the $n_j$ are all distinct. Hence $K_2 = Q_{n_N}$ and $K_1 = F$, so $H=G$. 
\end{proof}

Up to a reordering of the $p_i$, there exist $k,m,n \in \mathbb{N}$ with $k+m+n=N$ such that: 
\begin{itemize}
\item[(i)] $\pi_{p_i} (H) = Q_{n_i}$ for $1 \leq i \leq k$; 

\item[(ii)] $\pi_{p_i} (H)$ is metabelian for $k+1 \leq i \leq k+m$; 

\item[(iii)] $\pi_{p_i} (H) \lneq Q_{n_i}$ is not metabelian for $k+m+1 \leq i \leq N$. 

\end{itemize}

Let $C_2 , \gamma > 0$ be constants satisfying the conditions of Lemma \ref{noncontypes}. 
For $M$ sufficiently large, by the admissibility hypothesis and Proposition \ref{subgroupclassification}, 
$\lvert \pi_{p_i} (H) \rvert \leq \lvert Q_{n_i} \rvert^{\frac{1}{C_2}}$ for $k+m+1 \leq i \leq N$. 
Moreover by Corollary \ref{subgrpsnotontofactors}, at least one of $m,n$ is non-zero. 
\\ \\
Write $F_1 = \prod_{i=1} ^k p_i$, $F_2 = \prod_{i=k+1} ^{k+m} p_i$, $F_3 = \prod_{i=k+m+1} ^N p_i$, 
so that $f = F_1 \cdot F_2 \cdot F_3$. 
Applying Lemma \ref{CRT} 
with $f$ replaced by $F_1$, $F_2$, $F_3$ respectively, we have: 

\begin{lem}
\begin{itemize}
\item[(i)] $\pi_{F_1} (H) = \prod_{i=1} ^k SL_2 (p^{n_i})$. 
\item[(ii)] $\pi_{F_2} (H)$ is metabelian. 
\item[(iii)] $\lvert \pi_{F_3} (H) \rvert \leq \lvert \pi_{F_3} (G) \rvert^{\frac{1}{C_2}}$. 
\end{itemize}
\end{lem}

Finally, we are ready to complete: 

\begin{proof}[Proof of Proposition \ref{nonconpropn}]
$\lvert H \rvert \geq \lvert \pi_{F_1} (H) \rvert = \lvert \pi_{F_1} (G) \rvert$, so: 
\begin{center}
$\lvert G:H \rvert \leq \lvert G \rvert / \lvert \pi_{F_1} (G) \rvert 
= \lvert \pi_{F_2 F_3} (G) \rvert$. 
\end{center}

\begin{itemize}
\item[Case 1:] $\deg(F_2) \geq \deg(F_3)$: 

We have $\lvert \pi_{F_2 F_3}(G) \rvert^{\frac{1}{2}} 
\ll_p \lvert \pi_{F_2}(G) \rvert 
\leq \lvert \pi_{F_2 F_3}(G) \rvert$. 

By Lemma \ref{girthlem}, $\girth(\pi_{F_2} (G), S) \geq \frac{1}{\tilde{D}} \deg(F_2) \geq \frac{1}{3 \tilde{D} \log(p)} \log \lvert \pi_{F_2} (G) \rvert$, so that 
for $C_5 \log \lvert \pi_{F_2} (G) \rvert \leq 2l \leq \frac{1}{96 \tilde{D} \log(p)} \log \lvert \pi_{F_2} (G) \rvert$, 
by Lemma \ref{noncontypes} (i), 
\begin{center}
$\mu_S ^{(2l)} (H) \leq \mu_S ^{(2l)} (\pi_{F_2}(H)) 
\leq  \lvert \pi_{F_2}(G) \rvert^{-\gamma} 
\ll_p \lvert \pi_{F_2 F_3}(G) \rvert^{\frac{-\gamma}{2}}
\leq \lvert G:H \rvert^{\frac{-\gamma}{2}}$
\end{center}
and 
\begin{center}
$2l \leq \frac{1}{96 \tilde{D} \log(p)} \log \lvert \pi_{F_2} (G) \rvert \leq \frac{1}{96 \tilde{D} \log(p)} \log \lvert \pi_{F_2 F_3} (G) \rvert$. 
\end{center}

\item[Case 2:] $\deg(F_3) \geq \deg(F_2)$: 

We have $\lvert \pi_{F_2 F_3}(G) \rvert^{\frac{1}{2}} 
\ll_p \lvert \pi_{F_3}(G) \rvert \leq \lvert \pi_{F_2 F_3}(G) \rvert$. 

By Lemma \ref{girthlem}, 
$\girth(\pi_{F_3} (G), S) \geq \frac{1}{\tilde{D}} \deg(F_3) \geq \frac{1}{3 \tilde{D} \log(p)} \log \lvert \pi_{F_3} (G) \rvert$, 
so that for $C_5 \log \lvert \pi_{F_3} (G) \rvert \leq 2l \leq \frac{1}{3 \tilde{D} \log(p)} \log \lvert \pi_{F_3} (G) \rvert$, 
by Lemma \ref{noncontypes} (ii), 
\begin{center}
$\mu_S ^{(2l)} (H) \leq \mu_S ^{(2l)} (\pi_{F_3}(H)) 
\leq  \lvert \pi_{F_3}(G) \rvert^{-\gamma} 
\ll_p \lvert \pi_{F_2 F_3}(G) \rvert^{\frac{-\gamma}{2}}
\leq \lvert G:H \rvert^{\frac{-\gamma}{2}}$
\end{center}
and
\begin{center}
$2l \leq \frac{1}{3 \tilde{D} \log(p)} \log \lvert \pi_{F_3} (G) \rvert \leq \log \lvert \pi_{F_2 F_3} (G) \rvert$. 
\end{center}
\end{itemize}
The required result follows. 
\end{proof}

\section{Non-Concentration Results} \label{sieve}

\subsection{Two Different Sieves} \label{TwoSievesSection}

We start with a simple observation: 

\begin{lem} \label{simpleupperbound}
 Let $G$ be a discrete countable group; $H$ a finite group and $\phi : G \rightarrow H$ an epimorphism. 
 Let $\nu$ be a probability measure on $G$ and $X \subseteq G$. 
 Then $\nu (X) \leq (\phi \nu) (\phi(X)) \leq \lvert \phi (X) \rvert \cdot \max_{x \in X} (\phi \nu) (\phi(x))$. 
\end{lem}

Though straightforward, this bound can be very useful: 
when $\nu = \mu_S ^{(l)}$, for $S \subseteq G$ symmetric, with $\phi (S)$ generating $H$, 
then for $l$ sufficiently large and even, $\phi \nu$ is almost uniform on $H$, so that: 
\begin{equation} \label{measuresizebound}
(\phi \nu) (\phi (X)) \ll \lvert \phi (X) \rvert / \lvert H \rvert. 
\end{equation}
Moreover if $(H , \phi (S))$ is a good expander, 
equidistribution occurs for $l$ sufficiently \emph{small} 
that (\ref{measuresizebound}) gives a non-trivial lower bound on the rate at which $\mu_S ^{(l)}$ 
escapes from $X$. 

The present section contains two different instantiations of this philosophy for the group $SL_2 (\mathbb{F}_p [t])$, taking $(H,\phi)$ to be one of the congruence quotients from Theorem \ref{mainthm}. 
In the first of these it shall be sufficient to consider congruences modulo irreducible polynomials. 
We define, for $G$ a countable discrete group and $\nu_1 , \ldots , \nu_r$ 
finitely supported probability measures on $G$, the product measure $\times_{i=1} ^r \nu_i$ on $G$ by: 
\begin{center}
$(\times_{i=1} ^r \nu_i)(X) = \sum_{x \in X} \prod_{i=1} ^r \nu_i (x)$, for $X \subseteq G$. 
\end{center}

\begin{propn} \label{smallsieve}
Let $S \subseteq SL_2 (\mathbb{F}_p [t])$, $M > 0$ be as in Theorem \ref{mainthm}; 
let $(n_i)_i$ be as in Example \ref{smallsubfieldsap} (ii) 
and let $f_i \in \mathbb{F}_p [t]$ be irreducible of degree $n_i$. 
Let $X \subseteq SL_2 (\mathbb{F}_p [t])^r$ and suppose there exists $\alpha , C > 0$ 
such that for all $i$ sufficiently large, 
\begin{center}
$\lvert \pi_{f_i} (X) \rvert / \lvert Q_{n_i} \rvert^r \leq C p^{- \alpha n_i}$. 
\end{center}
Then there exist $C_1 (C,r) , C_2 (\alpha,p,S) > 0$ such that for all $l \in \mathbb{N}$, 
\begin{center}
$(\times_{i=1} ^r \mu_S ^{(l)})(X) \leq C_1 e^{- C_2 l}$. 
\end{center}
\end{propn}

\begin{proof}
By Theorem \ref{mainthm} and Lemma \ref{mixingbound}, 
there exists $c>0$ such that, for $i \geq i_0$, 
$l \geq c n_i$ and any $x \in Q_{n_i}$, 
$(\pi_{f_i} \mu_S ^{(l)}) (x) \leq 2 / \lvert Q_{n_i} \rvert$. 
Fix $\delta \in (0 , 1)$, so that for $l$ sufficiently large, 
$\exists i \geq i_0$ such that $l \geq c n_i \geq \delta l$. 
Then for $i$ sufficiently large, 
\begin{center}
$(\times_{j=1} ^r \mu_S ^{(l)})(X) \leq (\times_{j=1} ^r \pi_{f_i} \mu_S ^{(l)})( \pi_{f_i}(X))$ \\
$\leq 2^r \lvert \pi_{f_i} (X) \rvert / \lvert Q_{n_i} \rvert^r$ (by Lemma \ref{simpleupperbound}) \\
$\leq 2^r C p^{- \alpha n_i}$ (by hypothesis) \\
$\leq 2^r C e^{- \frac{\alpha \delta \log(p) l}{c}} $
\end{center}
as required. 
\end{proof}

Proposition \ref{smallsieve} is very useful for proving escape of the random walk 
from such subsets as proper algebraic subvarieties, 
which have small image in congruence quotients, as we shall see. 
Indeed, we have already implicitly used a form of Proposition \ref{smallsieve} 
in the proof of Theorem \ref{mainthm}, to establish non-concentration in subgroups. 
However, Proposition \ref{smallsieve} is powerless in the face of subsets $X$ whose images modulo $f_i$  
are of order $\sim \gamma \lvert Q_{n_i} \rvert$, for $\gamma \in (0,1)$, say. 
This difficulty may be partially resolved by considering, 
instead of individual congruence quotients $Q_{n_i}$, 
large products $Q_{n_i} \times \ldots \times Q_{n_{i+k}}$. 
The image of $X$ in such a quotient will be of order 
$\sim \gamma^k \lvert Q_{n_i} \rvert \cdots \lvert Q_{n_{i+k}} \rvert$, 
so by allowing $k$ to grow and applying Theorem \ref{mainthm}, 
we may recover a good non-concentration estimate. 
As discussed in the Introduction, Theorem \ref{mainthm} 
is not powerful enough to retain exponentially fast escape from such $X$. 
However we still have:

\begin{propn} \label{bigsieve}
Let $S \subseteq SL_2 (\mathbb{F}_p [t])$, $M>0$ be as in Theorem \ref{mainthm}; 
let $(n_i)_i$ be the sequence of all primes greater than $M$ (arranged in ascending order) 
and let $f_i \in \mathbb{F}_p [t]$ be irreducible of degree $n_i$. 
Let $X \subseteq SL_2 (\mathbb{F}_p [t])^r$ and suppose there exists $\gamma \in (0,1)$ 
and $i_1 \in \mathbb{N}$ such that for all $i \geq i_1$, 
\begin{center}
$\lvert \pi_{f_i} (X) \rvert / \lvert Q_{n_i} \rvert^r \leq \gamma$. 
\end{center}
Then there exist $C_1 (r), C_2 (\gamma,p,S) > 0$ such that for all $l \in \mathbb{N}$, 
\begin{center}
$(\times_{i=1} ^r \mu_S ^{(l)})(X) \leq C_1 e^{- C_2 \sqrt{l/\log(l)}}$. 
\end{center}
\end{propn}

\begin{proof}
Define $g_i = \prod_{k=i_2} ^{i_2 +i-1} f_k \in \mathbb{F}_p [t]$, with $i_2$ sufficiently large (to be determined). 
Then: 
\begin{center}
$\lvert  \pi_{g_i}(X) \rvert \leq \gamma^i \prod_{k=i_2} ^{i_2 +i-1} \lvert Q_{n_k} \rvert^r$
\end{center}
(provided $i_2 \geq i_1$). By Theorem \ref{mainthm}, there exists $c>0$ such that, 
provided $i_2$ is sufficiently large, for $l \geq c \sum_{k=i_2} ^{i_2 +i-1} n_k$ 
and for any $g \in SL_2 (\mathbb{F}_p [t] / (g_i))$, 
\begin{center}
$(\pi_{g_i} \mu_S ^{(l)})(g) \leq 2 / \prod_{k=i_2} ^{i_2 +i-1} \lvert Q_{n_k} \rvert$. 
\end{center}
For such $l$, 
\begin{center}
$(\times_{j=1} ^r \mu_S ^{(l)})(X) \leq (\times_{j=1} ^r \pi_{g_i} \mu_S ^{(l)})( \pi_{g_i}(X))$
\\
$\leq \lvert \pi_{g_i}(X) \rvert (2 / \prod_{k=i_2} ^{i_2 +i-1} \lvert Q_{n_k} \rvert)^r $
\\
$\leq 2^r \gamma^i$. 
\end{center}
Recalling that $n_k$ is of the order of $k \log (k)$, we have $\sum_{k=i_2} ^{i_2 +i-1} n_k \asymp i^2 \log(i)$. 
Choosing $i \asymp \sqrt{l/\log(l)}$, $l \gg i^2 \log(i)$ and the result follows. 
\end{proof}

\subsection{Escape from Subvarieties} \label{escape}

We are now ready to prove Theorem \ref{sieveupstairs}. In view of Proposition \ref{smallsieve}, 
it will suffice to bound the size of projections of subvarieties to congruence quotients. We use: 

\begin{thm}[Schwarz-Zippel \cite{LangWeil}] \label{schwarzzippel}
Let $\mathbb{F}$ be a finite field; $\overline{\mathbb{F}}$ be its algebraic closure. 
Let $V$ be an affine algebraic subvariety of $\mathbb{F}^d$, 
defined by $A$ polynomials in $\overline{\mathbb{F}} [x_1 , \ldots , x_d]$, 
each of total degree at most $B$. 
Then $\lvert V \rvert \ll_{A,B,d} \lvert \mathbb{F} \rvert^{\dim(V)}$. 
\end{thm}

\begin{proof}[Proof of Theorem \ref{sieveupstairs}]
$SL_2 ^r$ is irreducible of dimension $3r$, so by Theorem \ref{schwarzzippel}, 
\begin{center}
$\lvert \pi_{f_i}(V(F)) \rvert \ll_F p^{(3r-1)n_i} \ll p^{-n_i} \lvert Q_{n_i} \rvert^r$. 
\end{center}
The result now follows from Proposition \ref{smallsieve}. 
\end{proof}

\begin{ex} \label{subvarietyexamples}
Under the hypotheses of Theorem \ref{sieveupstairs}: 
\begin{itemize}
\item[(i)] Zero entries are rare: let $F_1 : \mathbb{M}_2 (\mathbb{F}_p [t]) \rightarrow \mathbb{F}_p [t]$ 
be given by  \linebreak$F_1 \left( \begin{array}{cc} a & b \\ c & d \end{array} \right) = abcd$. 
Then there exist $C_1 , C_2 > 0$ such that: 
\begin{center}
$\mu_S ^{(l)} (\lbrace g \in SL_2 (\mathbb{F}_p [t]) : g \text{ has a zero entry} \rbrace) = 
\mu_S ^{(l)} (V(F_1)) \leq C_1 e^{-C_2 l}$. 
\end{center}

\item[(ii)] Matrices with a particular trace are rare: 
fix $\alpha \in \mathbb{F}_p [t]$ and let  \linebreak$F_{\alpha} : \mathbb{M}_2 (\mathbb{F}_p [t]) \rightarrow \mathbb{F}_p [t]$ 
be given by $F_{\alpha} (A) = \tr(A)-\alpha$. Then there exist $C_1 , C_2 > 0$ such that: 
\begin{center}
$\mu_S ^{(l)} (\lbrace g \in SL_2 (\mathbb{F}_p [t]) : \tr(g)=\alpha \rbrace) = 
\mu_S ^{(l)} (V(F_{\alpha})) \leq C_1 e^{-C_2 l}$. 
\end{center}

\item[(iii)] Torsion elements are rare: 
Let $g \in SL_2 (\mathbb{F}_p [t])$. Conjugate $g$, possibly over a quadratic extension, 
to an upper triangular matrix  \linebreak$\tilde{g} = \left( \begin{array}{cc} a & b \\ 0 & a^{-1} \end{array} \right)$. 
Suppose there exists $n \in \mathbb{N}$ such that $g^n = I_2$. 
Then $a^n = 1$. This is only possible if $a$ lies in a quadratic extension of $\mathbb{F}_p$. 
In particular $\tr(g) \in \mathbb{F}_p$, so $g$ satisfies one of the bounded set of polynomials 
$F_{\alpha}$ as in (ii) above, for $\alpha \in \mathbb{F}_p$. 
Hence there exist $C_1 , C_2 > 0$ such that: 
\begin{center}
$\mu_S ^{(l)} (\lbrace g \in SL_2 (\mathbb{F}_p [t]) : g \text{ has finite order} \rbrace)
\leq \sum_{\alpha \in \mathbb{F}_p} \mu_S ^{(l)} (V(F_{\alpha})) \leq C_1 e^{-C_2 l}$. 
\end{center}

\item[(iv)] Elements fixing a point in the adjoint representation are rare: 
Recall that $SL_2 (\mathbb{F}_p [t])$ acts linearly on $\mathfrak{sl}_2 (\mathbb{F}_p [t])$ by conjugation. 
Given $g \in SL_2 (\mathbb{F}_p [t])$, 
let $\Ad(g) \in GL_3 (\mathbb{F}_p [t])$ be the matrix of the associated linear transformation 
with respect to some (fixed) $\mathbb{F}_p [t]$-basis for $\mathfrak{sl}_2 (\mathbb{F}_p [t])$. 
\\ \\
Now recall that, given polynomials $F_1 (X) , F_2 (X)$ over some field $K$, 
there is a polynomial function $Res(F_1 (X) , F_2 (X))$ of their coefficients 
(defined over $\mathbb{Z}$ and depending only on the degrees of $F_1 , F_2$) 
which vanishes precisely when $F_1 , F_2$ have a common root in some extension of $K$. 
In particular, $F(g) = Res (\chi_{\Ad(g)}(X),X-1)$ is a polynomial in the entries of $g$ 
which vanishes precisely when $g$ has a non-zero fixed point in $\mathfrak{sl}_2 (\mathbb{F}_p [t])$. 
Moreover $F(g)$ does not vanish identically on $SL_2 (\mathbb{F}_p [t])$: 
$F \left( \begin{array}{cc} 1+t & 2+t \\ t & 1+t \end{array} \right) \neq 0$, for instance. 
We conclude that there exist $C_1 , C_2 > 0$ such that: 
\begin{center}
$\mu_S ^{(l)} (\lbrace g \in SL_2 (\mathbb{F}_p [t]) : \exists X \in \mathfrak{sl}_2 (\mathbb{F}_p [t]) \setminus \lbrace 0 \rbrace \text{ s.t. } X^g=X \rbrace) \leq C_1 e^{-C_2 l}$. 
\end{center}
\end{itemize}
\end{ex}

\subsection{Squares in $SL_2 (\mathbb{F}_p [t])$ are Rare}

In this section we prove Theorem \ref{squaresarerare}. 
Let $X \subseteq SL_2 (\mathbb{F}_p [t])$ be the set of squares. 
In light of Proposition \ref{bigsieve}, 
it suffices to bound the sizes of images $\pi_{f_i} (X)$. 
We note some elementary facts about $SL_2 (Q)$, for $Q$ an arbitrary odd prime power. 
Let $D (Q) \leq SL_2 (Q)$ be the subgroup of diagonal matrices. 
Recall that $D (Q)$ is cyclic of order $Q-1$. 

\begin{lem} \label{SL2basics}
Let $g \in D(Q)$ be non-central in $SL_2 (Q)$. Then:
\begin{itemize}
\item[(i)] $C_{SL_2 (Q)} (g) = D(Q)$. 
\item[(ii)] $\lvert \ccl_{SL_2 (Q)} (g) \cap D(Q) \rvert = 2$. 
\item[(iii)] If $g$ is a square in $SL_2 (Q)$ then it is a square in $D(Q)$. 
\end{itemize}
\end{lem}

Now $2 \mid (Q-1)$, so the set of squares in $D(Q)$ is of order $\frac{Q-1}{2}$. 
 \linebreak$Z (SL_2 (Q)) = \lbrace \pm I_2 \rbrace$ consists of squares in $SL_2 (Q)$, 
so that by Lemma \ref{SL2basics} (iii), 
there is a subset $\lbrace g_i \rbrace_{i=1} ^{\frac{Q-1}{2}} \subseteq D(Q)$ 
consisting entirely of non-squares in $SL_2 (Q)$. 
If $g \in SL_2 (Q)$ is not a square, then $\ccl_{SL_2 (Q)} (g)$ consists entirely of non-squares, 
and by Lemma \ref{SL2basics} (i), $\lvert \ccl_{SL_2 (Q)} (g) \rvert = Q (Q + 1)$. Hence: 
\begin{center}
$\lvert \lbrace \text{non-squares in } SL_2 (Q) \rbrace \rvert 
\geq \lvert \bigcup_{i=1} ^{\frac{Q-1}{2}} \ccl_{SL_2 (Q)} (g_i) \rvert$
\\
$\geq \frac{1}{2} \sum_{i=1} ^{\frac{Q-1}{2}} \lvert \ccl_{SL_2 (Q)} (g_i) \rvert$ (by Lemma \ref{SL2basics} (ii))
\\
$\geq \frac{1}{4} (Q-1)Q(Q+1)$
\\
$ = \frac{1}{4} \lvert SL_2 (Q) \rvert$. 
\end{center}
Theorem \ref{squaresarerare} is now immediate from Proposition \ref{bigsieve}, taking $\gamma = \frac{3}{4}$. 

\subsection{Reducible Characteristic Polynomials in $SL_2 (\mathbb{F}_p [t])$ are Rare}

In this section we prove Theorem \ref{redcharpolysarerare}. 
Let $Y \subseteq SL_2 (\mathbb{F}_p [t])$ be the set of elements with reducible characteristic polynomial. 
Once again, we bound $\lvert \pi_{f_i} (Y) \rvert$. 
Let $g \in Y$ and let $f \in \mathbb{F}_p [t]$ be irreducible of degree $n$. 
Since $\chi_g \in \mathbb{F}_p [t] [X]$ splits over $\mathbb{F}_p [t]$, 
$\chi_{\pi_f (g)} \in \mathbb{F}_{p^n} [X]$ splits over $\mathbb{F}_{p^n}$. 
Let $Q$ be an arbitrary odd prime power. 
It will suffice to bound the set of elements $g \in SL_2 (Q)$ 
with reducible characteristic polynomial. We distinguish two cases and prove exponential decay in each:
\begin{itemize}
\item[Case 1:] $\tr(g) \neq \pm 2$. 
\\
$\chi_g$ does not have a repeated root, so is diagonalisable in $SL_2 (Q)$. 
Hence there exists non-central $h \in D(Q)$ such that $\ccl_{SL_2 (Q)} (g) = \ccl_{SL_2 (Q)} (h)$. 
There are $Q-3$ non-central elements $h \in D(Q)$, 
and each has conjugacy class in $SL_2 (Q)$ of order $Q(Q+1)$, 
by Lemma \ref{SL2basics} (i). 
Therefore the number of non-central diagonalisable elements $g$ is at most: 
\begin{center}
$\lvert \bigcup_{h \in D(Q) \setminus Z(SL_2 (Q))} \ccl_{SL_2 (Q)} (h) \rvert 
\leq \frac{1}{2} \sum_{h \in D(Q) \setminus Z(SL_2 (Q))} \lvert \ccl_{SL_2 (Q)} (h) \rvert$ 
\\
(by Lemma \ref{SL2basics} (ii))
\\
$\leq \frac{1}{2} (Q-3)Q(Q+1)$ 
\\
$\leq \frac{1}{2} \lvert SL_2 (Q) \rvert$. 
\end{center}

\item[Case 2:] $\tr(g) = \pm 2$ is immediate from Example \ref{subvarietyexamples}. 

\end{itemize}
Theorem \ref{redcharpolysarerare} follows from Proposition \ref{bigsieve}, with any $\gamma > \frac{1}{2}$.


\begin{thebibliography}{99}
\bibitem{BoGa1} J. Bourgain, A. Gamburd. Uniform Expansion Bounds for Cayley Graphs of $SL_2 (\mathbb{F}_p)$. 
Annals of Mathematics, \textbf{167} (2008), 625-642 
\bibitem{UnifTits} E. Breuillard. A Strong Tits Alternative. arXiv:0804.1395 
\bibitem{BreSurvey} E. Breuillard. Approximate subgroups and super-strong approximation. arXiv:1407.3158 
\bibitem{BrGrTa} E. Breuillard, B. Green, T. Tao. Approximate Subgroups of Linear Groups. 
Geom. Funct. Anal. Vol. 21 (2011) 774-819
\bibitem{BrGrGuTa} E. Breuillard, B. Green, R. Guralnick, T. Tao. 
Expansion in finite simple groups of Lie type. arXiv:1309.1975 
\bibitem{Dinai} O. Dinai. Expansion properties of finite simple groups. 
Ph. D. thesis. The Hebrew University of Jerusalem (2009) 
\bibitem{GHSSV} A. Gamburd, S. Hoory, M. Shahshahani, A. Shalev, B. Vir\'{a}g. 
On the girth of random Cayley graphs. Random Structures and Algorithms 35 (2009), no.1, 100-117 
\bibitem{Gowers} W.T. Gowers. Quasirandom Groups. Combinatorics, Probability and Computing, Volume 17, Issue 03, May 2008, pp 363-387 
\bibitem{OHKing} O.H. King. The subgroup structure of finite classical groups in terms of geometric configurations. Surveys in combinatorics, 2005. Edited by B. S. Webb. 
\bibitem{LandSeitz} V. Landazuri, G. Seitz. On the minimal degrees of projective representations of the finite Chevalley groups. J. Algebra 32 (1974), 418-443
\bibitem{LangWeil} S. Lang, A. Weil. Number of points of varieties in finite fields.  Amer. J. Math. \textbf{76}, (1954), 819-827 
\bibitem{LuMe} A. Lubotzky, C. Meiri. Sieve methods in group theory I: Powers in linear groups. 
J. Amer. Math. Soc. 25 (2012), 1119-1148 
\bibitem{PySz} L. Pyber, E. Szabo. Growth in finite simple groups of Lie type of bounded rank. arXiv:1005.1858 
\bibitem{SaGoVar} A. Salehi Golsefidy, P. P. Varj\'{u}. Expansion in Perfect Groups. 
Geom. Funct. Anal., 22 (6), 1832-1891 (2012) 
\bibitem{SarXue} P. Sarnak, X. Xue. Bounds for multiplicities of automorphic representations. 
Duke Math. J. \textbf{64} (1991), 207-227
\bibitem{TaoBook} T. Tao. Expansion in finite simple groups of Lie type. 
http://terrytao.files.wordpress.com/2014/04/expander-book.pdf 
\bibitem{TaoBSG} T. Tao. Product set estimates for non-commutative groups. 
Combinatorica, September 2008, Volume 28, Issue 5, pp 547-594 
\bibitem{Varju} P. P. Varj\'{u}. Expansion in $SL_d (\mathcal{O}_K / I)$, $I$ square free. 
J. Eur. Math. Soc. (JEMS) 14 (2012), no. 1, 273-305
\end{thebibliography}
\end{document}